\documentclass{amsart}
\usepackage{amsmath,amssymb}
\usepackage{bm}
\usepackage{graphicx}
\usepackage{float}
\usepackage{ascmac}
\usepackage{fancybox}
\usepackage{tabularx}
\usepackage{multicol}
\usepackage{enumerate}
\usepackage{slashed}
\usepackage{mathrsfs}

\newtheorem{dfn}{Definition}[section]
\newtheorem{thm}[dfn]{Theorem}
\newtheorem{prop}[dfn]{Proposition}

\newtheorem{cor}[dfn]{Corollary}
\newtheorem{remark}[dfn]{Remark}
\newtheorem{example}[dfn]{Example}

\numberwithin{equation}{section}

\DeclareMathOperator*{\divergence}{div}

\title[Stability estimates for a hyperbolic inverse problem]{An inverse hyperbolic obstacle problem}

\author{Mourad Choulli}
\address{Universit\'{e} de Lorraine, 34 cours L\'{e}opold, 54052 Nancy cedex, France}
\email{mourad.choulli@univ-lorraine.fr}

\author{Hiroshi Takase}
\address{Institute of Mathematics for Industry, Kyushu University, 744 Motooka, Nishi-ku, Fukuoka 819-0395, Japan}
\email{htakase@imi.kyushu-u.ac.jp}

\date{\today}

\begin{document}
\begin{abstract}
We establish H\"older stability of an inverse hyperbolic obstacle problem. Mainly, we study the problem of reconstructing an unknown function defined on the boundary of the obstacle from two measurements taken on the boundary of a domain surrounding the obstacle. 
\end{abstract}
\maketitle

\section{Introduction and main result}

Let $n\ge 2$ be an integer. Throughout this text, we use the Einstein summation convention for quantities with indices. If in any term appears twice, as both an upper and lower index, that term is assumed to be summed from $1$ to $n$.

Let $(g_{k\ell})\in W^{2,\infty}(\mathbb{R}^n;\mathbb{R}^{n\times n})$ be a symmetric matrix-valued function satisfying
\[
g_{k\ell}(x)\xi^k\xi^\ell\ge \kappa|\xi|^2\quad x,\xi \in\mathbb{R}^n,
\]
where $\kappa >0$ is a constant. Note that $(g^{k\ell})$  the matrix inverse to $g$ is uniformly positive definite as well. Recall that the Laplace-Beltrami operator associated with  the metric tensor $g=g_{k\ell}dx^k\otimes dx^\ell$ is given by
\[
\Delta_g u:=\left(1/\sqrt{|g|}\right)\partial_k\left(\sqrt{|g|}g^{k\ell}\partial_\ell u\right),
\]
where $|g|=\mbox{det} (g)$.

For convenience, we recall the following usual notations, 
\begin{align*}
&\langle X,Y\rangle=g_{k\ell}X^kY^\ell,\quad X=X^k\frac{\partial}{\partial x^k},\; Y=Y^k\frac{\partial}{\partial x^k},
\\
&\nabla_gw=g^{k\ell}\partial_k w\frac{\partial}{\partial x^\ell},
\\
&|\nabla_gw|_g^2=\langle\nabla_gw,\nabla_g\overline{w}\rangle=g^{k\ell}\partial_k w\partial_\ell \overline{w}.
\end{align*}

Also, recall that the Christoffel symbols $\Gamma_{k\ell}^m$, $1\le k,\ell,m\le n$, are given as follows
\[
\Gamma_{k\ell}^m:=(1/2)g^{jm}(\partial_kg_{j \ell}+\partial_\ell g_{jk}-\partial_jg_{k\ell}),
\]
and the  Hessian matrix $\nabla_g^2w$ is defined by
\[
(\nabla_g^2w)_{k\ell}:= \partial_{k}\partial_\ell w-\Gamma_{k\ell}^m\partial_mw.
\]

If $dx$ is the Lebesgue measure on $\mathbb{R}^n$, set $dV_g=\sqrt{|g|}dx$. Let $O$ be an open subset of $\mathbb{R}^n$. We assume hereinafter that $L^2(O)$, $H^1(O)$ and $H^2(O)$ are endowed respectively with the norms
\begin{align*}
&\|w\|_2:=\left(\int_O |w|^2dV_g\right)^{1/2},
\\
&\|w\|_{1,2}:= \left(\int_O \left[|w|^2+|\nabla_gw|_g^2\right]dV_g\right)^{1/2},
\\
&\|w\|_{2,2}:= \left(\int_O \left[|w|^2+|\nabla_gu|_g^2+|\nabla_g^2w| ^2_g\right]dV_g\right)^{1/2}.
\end{align*}
Here
\[
|\nabla_g^2w| ^2_g:=g^{k_1\ell_1}g^{k_2\ell_2}(\nabla_g^2w)_{k_1k_2}(\nabla_g^2\overline{w})_{\ell_1\ell_2}.
\]
In the rest of this text, $\phi\in C^3(\mathbb{R}^n)$ satisfies the following properties
\[
\inf_{\{\phi> 0\}}|\nabla_g\phi |_g>0 ,\quad \inf_{\{\phi> 0\}}\inf_{|\xi|=1}\nabla_g ^2\phi(\xi,\xi)>0
\]
and $\{\phi<0\}\subset\mathbb{R}^n$ is bounded. Let us give an example of such a function.

\begin{example}{\rm
We denote the open unit ball of $\mathbb{R}^n$ by $\mathbb{B}$ and set $\mathbb{B}_r=r\mathbb{B}$, $r>0$. For fixed $r_0>0$, let $\phi=|x|^2-r_0^2$. Furthermore, assume that $g=\mathbf{I}$ in $\mathbb{B}_{r_1}^c:=\mathbb{R}^n\setminus \mathbb{B}_{r_1}$, where $r_1>r_0$ and  $\mathbf{I}$ is the identity matrix of $\mathbb{R}^n$. Then, we have $|\nabla_g\phi(x)|_g=2|x|\ge 2r_1$ if $x\in \mathbb{B}_{r_1}^c$ and, since $g^{-1}$ is uniformly positive definite, we have
\[
|\nabla_g\phi(x)|_g>0,\quad x\in \overline{\mathbb{B}_{r_1}}\setminus \mathbb{B}_{r_0}.
\]
On the other hand, as $\nabla_g^2\phi=2\mathbf{I}$ in $\mathbb{B}_{r_1}^c$, we see that $\nabla_g^2\phi\ge \mathbf{I}$ in $\mathbb{B}_{r_0}^c$ provided that the norm of $g$ in $W^{1,\infty}(\overline{\mathbb{B}_{r_1}}\setminus \mathbb{B}_{r_0};\mathbb{R}^{n\times n})$ is sufficiently small.
}\end{example}

Before stating our main result, we introduce new definitions. Let $T>0$, $B:=\{\phi<0\}$, $\Gamma:=\{ \phi=0\}$ ($=\partial B$), $U:=\mathbb{R}^n\setminus \overline{B}$ and $\Sigma :=\Gamma\times (0,T)$.

Since we will consider an IBVP with nonhomogenous boundary condition, we introduce the space
\[
\mathcal{Y}:=\{(f,u_0,u_1)=(H_{|\Sigma},H(\cdot,0),\partial_tH(\cdot,0));  H\in \mathcal{Z}\},
\]
where 
\[
\mathcal{Z}:=W^{3,1}((0,T);L^2(U))\cap W^{2,1}((0,T);H^1(U))\cap W^{1,1}((0,T);H^2(U)).
\]
We equip $\mathcal{Z}$ with its natural norm denoted by $\|\cdot\|_{\mathcal{Z}}$, and we observe that $\mathcal{Z}$ is continuously embedded in $\mathcal{X}$, where
\[
\mathcal{X}:=C^2([0,T];L^2(U))\cap C^1([0,T];H^1(U))\cap C([0,T];H^2(U))
\] 
with its natural norm that we denote by $\| \cdot\|_{\mathcal{X}}$.
Define
\[
\dot{\overbrace{(f,u_0,u_1)}}:=\{H\in  \mathcal{Z};\; (H_{|\Sigma},H(\cdot,0),\partial_tH(\cdot,0))=(f,u_0,u_1)\}.
\]
We endow $\mathcal{Y}$ with the quotient norm
\[
\|(f,u_0,u_1)\|_{\mathcal{Y}}:=\inf\{\|H\|_{\mathcal{Z}};\; H\in \dot{\overbrace{(f,u_0,u_1)}}\}.
\]

Let $p\in L^\infty(U)$ and consider the nonhomogenuos IBVP
\begin{equation}\label{exIBVP}
\begin{cases}
(\partial_t^2 -\Delta_g-p)u=0\quad \text{in}\; U\times(0,T),
\\
u_{|\Sigma}=f,
\\
(u(\cdot,0),\partial_tu(\cdot,0))=(u_0,u_1).
\end{cases}
\end{equation}

We prove in Appendix \ref{appA} that, for each $(f,u_0,u_1)\in \mathcal{Y}$, \eqref{exIBVP} admits a unique solution $u=u(f,u_0,u_1)\in \mathcal{X}$ and
\[
\|u\|_{\mathcal{X}}\le \mathbf{c} \|(f,u_0,u_1)\|_{\mathcal{Y}},
\]
where $\mathbf{c}=\mathbf{c}(g,\phi,p,T)>0$ is a constant.

Let $\Omega \Supset B$ be an arbitrary fixed $C^{0,1}$ bounded domain with boundary $\partial\Omega$. Henceforth, we assume that $D:=\Omega\setminus \overline{B}$ is connected. We aim to establish H\"{o}lder stability of the determination of the unknown function $f$ from the measurements $u(f,0,0)_{|\Sigma_0}$ and $\partial_{\nu_g} u(f,0,0)_{|\Sigma_0}$, where $\Sigma_0 :=\partial\Omega \times(0,T)$. This problem can be seen as the determination of an unknown source emanating from the obstacle from two measurements performed far from the obstacle.

The same type of inverse problem was considered by Choulli-Takase \cite{CT2024} for elliptic and parabolic equations. See for example Isakov \cite{Isakov2009} and Bourgeois-Ponomarev-Dard\'e \cite{Bourgeois2019} for an inverse obstacle problem, which determines interior regions from finite-time Cauchy data. Recently, the first author obtained a stability inequality for an inverse source problem for an obstacle problem for the wave equation. We refer to Choulli \cite{Choulli2024}. Other references have been mentioned in this work.

We define
\[
\mathscr{F}_0:=\left\{f=\mathfrak{a}\otimes\mathfrak{b};\; \mathfrak{a}\in H^{3/2}(\Gamma),\; \mathfrak{b}\in W^{3,1}((0,T)),\; \mathfrak{b}(0)=\mathfrak{b}'(0)=0\right\}.
\]
It is worth observing that
\[
\{(f,0,0);\; f\in \mathscr{F}_0\}\subset \mathcal{Y}.
\]

The following additional notations will be used in the sequel
\begin{align*}
&\nu_g=(\nu_g)^k\frac{\partial}{\partial x^k},\quad (\nu_g)^k=\frac{g^{k\ell}\nu_\ell}{\sqrt{g^{\ell_1\ell_2}\nu_{\ell_1}\nu_{\ell_j}}},
\\
&\partial_{\nu_g}w=\langle\nu_g,\nabla_g w\rangle,
\end{align*}
where $\nu$ denotes the outer unit normal to $\partial D$. Also, define  $\nabla_{\tau_g} w$ by
\[
\nabla_{\tau_g} w=\nabla_g w-(\partial_{\nu_g} w)\nu_g.
\]
We check that 
\[
|\nabla_g w|_g^2=|\nabla_{\tau_g} w|_g^2+|\partial_{\nu_g} w|^2.
\]

Let $\alpha>0$ and $\beta>0$ be fixed and define the admissible set of the unknown function $f$ by
\[
\mathscr{F}:=\left\{f=\mathfrak{a}\otimes\mathfrak{b}\in \mathscr{F}_0; \; \|\mathfrak{a}\|_{L^2(\Gamma)}\ge\alpha\; \mbox{and}\; \|\nabla_{\tau_g}\mathfrak{a}\|_{L^2(\Gamma)}\le \beta\right\},
\]
where 
\[
\|\mathfrak{a}\|_{L^2(\Gamma)}:=\left(\int_\Gamma|\mathfrak{a}|^2dS_g\right)^{1/2}\quad \mathrm{and}\quad \|\nabla_{\tau_g}\mathfrak{a}\|_{L^2(\Gamma)}:=\left(\int_\Gamma|\nabla_{\tau_g}\mathfrak{a}|_g^2dS_g\right)^{1/2}. 
\]
Here $dS_g:=\sqrt{|g|}dS$, where $dS$ denotes the surface measure on $\partial D$.

Next, for $f\in \mathscr{F}_0$, let $u(f):=u(f,0,0)$,
\[\mathcal{D}(f):=\|u(f)\|_{H^1(\Sigma)}+\|\partial_{\nu_g} u(f)\|_{L^2(\Sigma)}\]
and
\[\mathcal{D}_0(f):=\|u(f)\|_{H^1(\Sigma_0)}+\|\partial_{\nu_g} u(f)\|_{L^2(\Sigma_0)}.
\]
Here $L^2(\Sigma)$ and $H^1(\Sigma)$ are endowed respectively with the norms
\begin{align*}
&\|w\|_{L^2(\Sigma)}:=\left(\int_\Sigma |w|^2dS_gdt\right)^{1/2},
\\
&\|w\|_{H^1(\Sigma)}:=\left(\int_\Sigma \left[|w|^2+|\partial_tw|^2+|\nabla_{\tau_g}w|_g^2\right]dS_gdt\right)^{1/2},
\end{align*}
and similarly for $L^2(\Sigma_0)$ and $H^1(\Sigma_0)$.

Finally, define $\mathfrak{m}:=\max_{\overline{D}}\phi$, $\delta:=\inf_{D}|\nabla_g \phi|_g$ and $T^\ast:=\mathfrak{m}/\delta$.

\begin{thm}\label{exHolder}
Assume that $T>T^\ast$. Let $\tau \in (0,T-T^\ast)$ and $\zeta=(g,\phi,p,T,\tau,\beta/\alpha)$. Then there exist constants $C=C(\zeta)>0$ and $\theta=\theta(\zeta)\in(0,1)$ such that for any $f\in\mathscr{F}$ we have
\[
\|f\|_{H^1(\Gamma\times(0,\tau))}
\le C\left(\mathcal{D}_0(f)+\mathcal{D}(f)^{1-\theta}\mathcal{D}_0(f)^\theta\right).
\]
\end{thm}

Note that Theorem \ref{exHolder} quantifies the unique determination of $f$ from the Cauchy data $\mathcal{D}_0(f)$. The stability inequality in Theorem \ref{exHolder} is of H\"older type. Indeed, if we assume the a priori bound $\max(\mathcal{D}_0(f),\mathcal{D}(f))\le \varkappa$, for some $\varkappa >0$, then the inequality in Theorem \ref{exHolder} yields
\[
\|f\|_{H^1(\Gamma\times(0,\tau))}\le C\varkappa^{1-\theta}\mathcal{D}_0(f)^\theta.
\]

%
%

In the case where $\mathfrak{b}$ is known we have the following result.

\begin{cor}
Assume that $T>T^\ast$. Let $\tau \in (0,T-T^\ast)$ and $\zeta=(g,\phi,p,T,\tau,\beta/\alpha)$. Then there exist constants $C=C(\zeta)>0$ and $\theta=\theta(\zeta)\in(0,1)$ such that for any $f=\mathfrak{a}\otimes \mathfrak{b}\in\mathscr{F}$ we have
\[
\|\mathfrak{a}\|_{H^1(\Gamma)}
\le C\|\mathfrak{b}\|_{L^2((0,\tau))}^{-1}\left(\mathcal{D}_0(f)+\mathcal{D}(f)^{1-\theta}\mathcal{D}_0(f)^\theta\right),
\]
provided that $\|\mathfrak{b}\|_{L^2((0,\tau))}>0$.
\end{cor}

To our knowledge, there are no other results in the literature comparable to those established in the present work.

\begin{remark}{\rm
We restrict our analysis to the case where $f\in\mathscr{F}_0$ can be written in the form of variable separation, for the sake of brevity. For more general cases, by replacing $\mathscr{F}_0$ with
\[
\mathscr{F}_0':=\left\{f=H_{|\Sigma};\; H\in\mathcal{Z},\; H(\cdot,0)=\partial_t H(\cdot,0)=0\right\}
\]
and $\mathscr{F}$ with
\[
\mathscr{F}':=\left\{f\in \mathscr{F}_0'; \; \|\nabla_{\tau_g}f(\cdot,t)\|_{L^2(\Gamma)}\le M\|f(\cdot,t)\|_{L^2(\Gamma)},\; t\in[0,T]\right\},
\]
where $M>0$ is a given constant, the same result as in Theorem \ref{exHolder} holds by replacing $\beta/\alpha$ with $M$.
}\end{remark}

\section{Proof of Theorem \ref{exHolder}}

\subsection{Preliminaries}

Recall that $D=\Omega \setminus \overline{B}$, $\Sigma=\partial B\times(0,T)$ and $\Sigma_0=\partial\Omega\times(0,T)$. Let $Q:=D\times (0,T)$, $\Upsilon :=\partial D\times (0,T)=\Sigma\cup\Sigma_0$, $0<\gamma <\delta$ and
\[
\varphi(x,t):=\phi(x)-\gamma t.
\]

The proof of Theorem \ref{exHolder} is based on the following Carleman inequality.

\begin{prop}\label{global_Carleman_estimate}
Let $\zeta_0=(g,\phi,p,\gamma)$. There exist constants $s_\ast=s_\ast (\zeta_0)>0$ and $C=C(\zeta_0)>0$ such that for any $s\ge s_\ast$ and $u\in H^2(Q)$ satisfying 
$u=\partial_t u=0$ on $D\times\{0\}$ we have
\begin{align*}
&C\left(\int_{Q} e^{2s\varphi}(s|\nabla_g u|_g^2+s^3|u|^2)dV_gdt+\int_{\Sigma}e^{2s\varphi}(s|\partial_t u|^2+s|\partial_{\nu_g}u|^2+s^3|u|^2)dS_gdt\right)
\\
&\hskip 2cm\le \int_{Q} e^{2s\varphi}|(\partial_t^2-\Delta_g-p)u|^2dV_gdt+\int_{\Sigma}e^{2s\varphi}s|\nabla_{\tau_g} u|_g^2dS_gdt
\\
&\hskip 4cm+\int_{\Sigma_0} e^{2s\varphi}(s|\partial_t u|^2+s|\nabla_g u|_g^2+s^3|u|^2)dS_gdt
\\
&\hskip 5cm+\int_{D\times\{T\}}e^{2s\varphi}(s|\partial_t u|^2+s|\nabla_g u|_g^2+s^3|u|^2)dV_g.
\end{align*}
\end{prop}

Note that only the integral term for the tangential derivative of $u$ on $\Sigma$ appears in the right-hand side. To our knowledge, this is the first Carleman estimate of this type for hyperbolic operators (see for instance \cite[Theorem 4.1]{Yamamoto2017}, \cite[Theorem 7.1]{Choulli}, \cite[Theorem 4.1]{Fu2019} and \cite[Proposition 3.1]{Takase2023}). Although the proof is quite technical, we give the full proof in Section \ref{proof_Carleman_estimate} for the sake of self-containedness.

We also need an energy inequality. To this end, set for any $u\in H^2(Q)$
\[
\mathcal{E}(u)(t):=\int_{D\times\{t\}}\left(|\partial_t u|^2+|\nabla_g u|_g^2+|u|^2\right)dV_g.
\]
We check that $\mathcal{E}(u)\in W^{1,1}((0,T);\mathbb{R})$.

\begin{prop}\label{energy_estimate}
For any $t\in[0,T]$ and $u\in H^2(Q)$ satisfying $(\partial_t^2-\Delta_g-p)u=0$ in $Q$ and $u=\partial_t u=0$ on $D\times\{0\}$ we have
\[
\mathcal{E}(u)(t)\le e^{(1+\|p\|_{L^\infty(D)})t} \left(\|u\|_{H^1(\Upsilon)}^2+\|\partial_{\nu_g} u\|_{L^2(\Upsilon)}^2\right).
\]
\end{prop}

The proof of this proposition will be given in Appendix \ref{appB}.

\subsection{Proof of Theorem \ref{exHolder}}

\begin{proof}
Let $\tau\in (0,T-T^\ast)$ and choose $\gamma\in(0,\delta)$ in such a way that $\mathfrak{m}-\gamma(T-\tau)<0$ and define
\[
\mu:=-\mathfrak{m}+\gamma(T-\tau)>0.
\]

Let $f=\mathfrak{a}\otimes \mathfrak{b}\in\mathscr{F}$ and $u=u(f,0,0)$. Then we have
\[
|\nabla_{\tau_g} u|_g=|\nabla_{\tau_g} f|_g=|\nabla_{\tau_g}\mathfrak{a}|_g|\mathfrak{b}|\quad \mbox{on}\; \Sigma
\]
and
\[
\|\nabla_{\tau_g}\mathfrak{a}\|_{L^2(\Gamma)}\le (\beta/\alpha)\|\mathfrak{a}\|_{L^2(\Gamma)}.
\]

Henceforth, $C=C(\zeta)>0$ and $c=c(\zeta)>0$ denote generic constants. For simplicity, we use in the sequel the following notations
\[\mathcal{D}:= \mathcal{D}(f)\; \left(=\|u(f)\|_{H^1(\Sigma)}+\|\partial_{\nu_g} u(f)\|_{L^2(\Sigma)}\right)\]
and
\[\mathcal{D}_0:=\mathcal{D}_0(f)\; \left(=\|u(f)\|_{H^1(\Sigma_0)}+\|\partial_{\nu_g} u(f)\|_{L^2(\Sigma_0)}\right).
\]

By $\varphi_{|\Sigma}=-\gamma t$ and applying Proposition \ref{global_Carleman_estimate} to $u$ yields
\begin{align*}
&C\int_\Sigma e^{2s\varphi}(s|\partial_t f|^2+s^3|f|^2)dS_gdt
\\
&\hskip .5cm \le \int_{\Sigma_0} e^{2s\varphi}(s|\partial_t u|^2+s|\nabla_g u|_g^2+s^3|u|^2)dS_gdt+\int_\Sigma e^{2s\varphi}s|\nabla_{\tau_g} f|_g^2dS_gdt
\\
&\hskip 2cm +\int_{D\times\{T\}}e^{2s\varphi}(s|\partial_t u|^2+s|\nabla_g u|_g^2+s^3|u|^2)dV_g
\\
&\hskip .5cm \le s^3e^{cs}\mathcal{D}_0^2+\int_0^T s e^{2s\varphi}|\mathfrak{b}|^2\int_\Gamma |\nabla_{\tau_g}\mathfrak{a}|_g^2dS_gdt+s^3e^{2(\mathfrak{m}-\gamma T)s}\mathcal{E}(u)(T)
\\
&\hskip .5cm \le s^3e^{cs}\mathcal{D}_0^2+(\beta/\alpha)^2\int_\Sigma e^{2s\varphi}s|f|^2dS_gdt+s^3e^{2(\mathfrak{m}-\gamma T)s}\mathcal{E}(u)(T),\quad s\ge s_\ast,
\end{align*}
where $s_\ast$ is as in Proposition \ref{global_Carleman_estimate}. Upon modifying $s_\ast$, we  may and do assume that $Cs^3-(\beta/\alpha)^2s>Cs^3/2$. In this case, we have
\[
C\int_\Sigma e^{2s\varphi}(|\partial_t f|^2+|f|^2)dS_gdt\le e^{cs}\mathcal{D}_0^2+s^{2}e^{2(\mathfrak{m}-\gamma T)s}\mathcal{E}(u)(T),\quad s\ge s_\ast.
\]
Since
\begin{align*}
\int_\Sigma e^{2s\varphi}(|\partial_t f|^2+|f|^2)dS_gdt&\ge\int_{\Gamma\times(0,\tau)}e^{2s\varphi}(|\partial_t f|^2+|f|^2)dS_gdt
\\
&\ge e^{-2\gamma\tau s}\|f\|_{H^1((0,\tau);L^2(\Gamma))}^2,
\end{align*}
we obtain by replacing the constant $c+2\gamma\tau$ with $c$
\[
C\|f\|_{H^1((0,\tau);L^2(\Gamma))}^2\le e^{cs}\mathcal{D}_0^2+s^2e^{-2\mu s}\mathcal{E}(u)(T),\quad s\ge s_\ast.
\]
By applying Proposition \ref{energy_estimate} to the second term of the right-hand side yields
\[
\mathcal{E}(u)(T)\le e^{(1+\|p\|_{L^\infty(D)})T}\left(\mathcal{D}_0^2+\mathcal{D}^2\right).
\]
Whence, by modifying the constant $c$, we obtain
\begin{align*}
C\|f\|_{H^1((0,\tau);L^2(\Gamma))}&\le e^{cs}\mathcal{D}_0+se^{-\mu s}\mathcal{D}
\\
&\le e^{cs}\mathcal{D}_0+e^{-(\mu/2) s}\mathcal{D},\quad s\ge s_\ast.
\end{align*}
By considering the new parameter $s':=s-s_\ast$ and modifying the constant $C$, we see that the above inequality holds for all $s\ge 0$. When $\mathcal{D}_0\ge\mathcal{D}$, $\|f\|_{H^1((0,\tau);L^2(\Gamma))}\le 2\mathcal{D}_0$ holds by choosing $s=0$. When $\mathcal{D}_0<\mathcal{D}$, we choose $s>0$ so that
\[
e^{cs}\mathcal{D}_0=e^{-(\mu/2) s}\mathcal{D},
\]
that is,
\[
s=\frac{1}{c+(\mu/2)}\log\frac{\mathcal{D}}{\mathcal{D}_0}.
\]
Thus, we get
\[
\|f\|_{H^1((0,\tau);L^2(\Gamma))}\le C\mathcal{D}^{1-\theta}\mathcal{D}_0^\theta
\]
for
\[\theta:=\frac{\mu}{2c+\mu}\in(0,1).\]
As
\begin{align*}
&\|f\|_{L^2(\Gamma\times(0,\tau))}=\|\mathfrak{a}\|_{L^2(\Gamma)}\|\mathfrak{b}\|_{L^2((0,\tau))},
\\
&\|\nabla_{\tau_g}f\|_{L^2(\Gamma\times(0,\tau))}\le (\beta/\alpha)\|\mathfrak{a}\|_{L^2(\Gamma)}\|\mathfrak{b}\|_{L^2((0,\tau))},
\end{align*}
the expected inequality follows.
\end{proof}

\section{Proof of Proposition \ref{global_Carleman_estimate}}\label{proof_Carleman_estimate}

\begin{proof}
It suffices to prove the inequality when $p=0$ due to the large parameter $s\ge s_\ast$.

Without loss of generality, we limit our proof to real-valued function. Also, for simplicity, we set $\square:=(\partial_t^2-\Delta_g)$ and $w':=\partial_t w$. Let $u\in H^2(Q)$, $z:=e^{s\varphi}u$, and 
\[
P_sz:=e^{s\varphi}\square(e^{-s\varphi}z).
\]
A direct calculation yields $P_sz=P_s^+z+P_s^- z$, where
\[
\begin{cases}
&P_s^+z:=\square z+s^2(|\varphi'|^2-|\nabla_g\varphi|_g^2)z,
\\
&P_s^-z:=-2s\varphi'z'+2s\langle\nabla_g\varphi,\nabla_g z\rangle-s\square\varphi z.
\end{cases}
\]

Hereinafter, the integrals on $Q=D\times(0,T)$ are with respect to the measure $dV_gdt$ and those on $\Upsilon=\partial D\times(0,T)$ are with respect to the surface measure $dS_gdt$. 

Set 
\[
(P_s^+z,P_s^-z)_g=:\sum_{k=1}^9 I_k,
\]
where
\begin{align*}
&I_1:=-2\int_{Q}s\varphi'z'z'',
\\
&I_2:=2\int_{Q} s\langle\nabla_g\varphi,\nabla_gz\rangle z'',
\\
&I_3:=-\int_{Q} s\square\varphi zz'',
\\
&I_4:=2\int_{Q}s\varphi' z'\Delta_g z,
\\
&I_5:=-2\int_{Q} s\Delta_gz\langle\nabla_g\varphi,\nabla_g z\rangle,
\\
&I_6:=\int_{Q} s\Delta_gz \square\varphi z,
\\
&I_7:=-2\int_{Q} s^3\varphi'(|\varphi'|^2-|\nabla_g\varphi|_g^2)zz',
\\
&I_8:=2\int_{Q} s^3(|\varphi'|^2-|\nabla_g\varphi|_g^2)\langle\nabla_g\varphi,\nabla_g z\rangle z,
\\
&I_9:=-\int_{Q} s^3\square\varphi(|\varphi'|^2-|\nabla_g\varphi|_g^2)|z|^2.
\end{align*}

Integrations by parts yield
\begin{align*}
&I_1=-\int_{Q}s\varphi'\partial_t(|z'|^2)=\int_{Q}s\varphi''|z'|^2-\int_D\left[s\varphi'|z'|^2\right]_0^T,
\\
&I_2=-\int_{Q}s\langle\nabla_g\varphi,\nabla_g|z'|^2\rangle+2\int_D\left[sz'\langle\nabla_g\varphi,\nabla_g z\rangle\right]_0^T
\\
&\hskip 2cm=\int_{Q}s\Delta_g\varphi|z'|^2-\int_\Upsilon s\partial_{\nu_g}\varphi |z'|^2+2\int_D\left[sz'\langle\nabla_g\varphi,\nabla_g z\rangle\right]_0^T,
\end{align*}

\begin{align*}
I_3&=\int_{Q}s\square\varphi |z'|^2+\int_{Q}s\partial_t\square\varphi z z'-\int_D\left[s\square\varphi zz'\right]_0^T
\\
&=\int_{Q}s\square\varphi|z'|^2-(1/2)\int_Qs\partial_t^2\square\varphi|z|^2-\int_D\left[s\square\varphi zz'\right]_0^T
\\
&\hskip 7cm+(1/2)\int_D\left[s\partial_t\square\varphi |z|^2\right]_0^T,
\end{align*}

\begin{align*}
I_4&=-\int_{Q}s\varphi'\partial_t(|\nabla_g z|_g^2)+2\int_\Upsilon s\varphi'\partial_{\nu_g} z z'
\\
&=\int_{Q}s\varphi''|\nabla_g z|_g^2+2\int_\Upsilon s\varphi'\partial_{\nu_g} z z'-\int_D\left[s\varphi'|\nabla_g z|_g^2\right]_0^T,
\end{align*}

\begin{align*}
I_5&=2\int_{Q}s\langle\nabla_g z,\nabla_g\langle\nabla_g\varphi,\nabla_g z\rangle\rangle-2\int_\Upsilon s\partial_{\nu_g} z\langle\nabla_g\varphi,\nabla_g z\rangle
\\
&=2\int_{Q}s\nabla_g^2\varphi(\nabla_g z,\nabla_g z)+\int_{Q}s\langle\nabla_g\varphi,\nabla_g(|\nabla_g z|_g^2)\rangle
\\
&\hskip 7cm-2\int_\Upsilon s\partial_{\nu_g} z\langle\nabla_g\varphi,\nabla_g z\rangle
\\
&=2\int_{Q}s\nabla_g^2\varphi(\nabla_g z,\nabla_g z)-\int_{Q}s\Delta_g\varphi|\nabla_g z|_g^2
\\
&\hskip 4cm-2\int_\Upsilon s\partial_{\nu_g} z\langle\nabla_g\varphi,\nabla_g z\rangle+\int_\Upsilon s\partial_{\nu_g}\varphi|\nabla_g z|_g^2,
\end{align*}
\[
I_6=-\int_{Q}s\square\varphi|\nabla_g z|_g^2-\int_{Q}s\langle\nabla_g\square\varphi,\nabla_g z\rangle z+\int_\Upsilon s\square\varphi\partial_{\nu_g} z z,
\]
\begin{align*}
I_7&=-\int_{Q}s^3\varphi'(|\varphi'|^2-|\nabla_g\varphi|_g^2)\partial_t(|z|^2)
\\
&=\int_{Q}s^3\partial_t(\varphi'(|\varphi'|^2-|\nabla_g\varphi|_g^2))|z|^2-\int_D\left[s^3\varphi'(|\varphi'|^2-|\nabla_g\varphi|_g^2)|z|^2\right]_0^T,\end{align*}

\begin{align*}
I_8&=\int_{Q}s^3(|\varphi'|^2-|\nabla_g\varphi|_g^2)\langle\nabla_g\varphi,\nabla_g(|z|^2)\rangle
\\
&=-\int_{Q}s^3\divergence((|\varphi'|^2-|\nabla_g\varphi|_g^2)\nabla_g\varphi)|z|^2+\int_\Upsilon s^3(|\varphi'|^2-|\nabla_g\varphi|_g^2)\partial_{\nu_g}\varphi |z|^2.
\end{align*}

Adding all of these together yields
\begin{align*}
&(P_s^+ z,P_s^- z)_g+\mathcal{B}=
\\
&\hskip .2cm \int_{Q}2s\left[\varphi''|z'|^2+\nabla_g^2\varphi(\nabla_g z,\nabla_g z)\right]-\int_{Q}s\left[\langle\nabla_g\square\varphi,\nabla_g z\rangle +(1/2)\partial_t^2\square\varphi z\right]z
\\
&\hskip .5cm+\int_{Q}\left[s^3\varphi'\partial_t(|\varphi'|^2-|\nabla_g\varphi|_g^2)-s^3\langle\nabla_g(|\varphi'|^2-|\nabla_g\varphi|_g^2),\nabla_g\varphi\rangle\right]|z|^2
\end{align*}
where

\begin{align*}
&\mathcal{B}:=\int_\Upsilon \left[s^3(|\nabla_g\varphi|_g^2-|\varphi'|^2)\partial_{\nu_g}\varphi\right] |z|^2
\\
&\hskip 1cm+\int_\Upsilon \left[-2s\varphi'z'+2s\langle\nabla_g\varphi,\nabla_g z\rangle-s\square\varphi z\right]\partial_{\nu_g}z
\\
&\hskip 1.5cm
+\int_\Upsilon \left[s\partial_{\nu_g}\varphi |z'|^2-s\partial_{\nu_g}\varphi|\nabla_g z|_g^2\right]
\\
&\hskip 2cm +\int_{D\times\{T\}}\left[(-s/2)\partial_t\square\varphi-s^3\varphi'(|\nabla_g\varphi|_g^2-|\varphi'|^2)\right]|z|^2
\\
&\hskip 2.5cm+\int_{D\times\{T\}}\left[-2s\langle\nabla_g\varphi,\nabla_g z\rangle+s\square\varphi z\right]z'
\\
&\hskip 5cm
+\int_{D\times\{T\}}s\varphi'\left[|z'|^2+|\nabla_g z|_g^2\right].
\end{align*}

Henceforth, $C=C(g,\phi,\gamma)>0$ and $s_\ast=s_\ast(g,\phi,\gamma)>0$ denote generic constants. Then we have
\[
\varphi''|z'|^2+\nabla_g^2\varphi(\nabla_g z,\nabla_g z)
=\nabla_g^2\phi(\nabla_g z,\nabla_g z)
\\
\ge C|\nabla_g z|_g^2,
\]
and
\begin{align*}
&\varphi'\partial_t(|\varphi'|^2-|\nabla_g\varphi|_g^2)-\langle\nabla_g(|\varphi'|^2-|\nabla_g\varphi|_g^2),\nabla_g\varphi\rangle
\\
&\hskip 2cm=2\nabla_g^2\phi(\nabla_g\phi,\nabla_g\phi) \ge C|\nabla_g\phi|_g^2.
\end{align*}
In consequence, we get
\[
C\int_{Q} (s|\nabla_g z|_g^2+s^3|z|^2)\le (P_s^+z,P_s^-z)_g+\int_{Q}s\left[\langle\nabla_g\square\varphi,\nabla_g z\rangle+(1/2)\partial_t^2\square\varphi z\right]z+\mathcal{B}.
\]

Set
\begin{align*}
&\mathcal{B}_{\Sigma_0}:=\int_{\Sigma_0} \left[s^3(|\nabla_g\varphi|_g^2-|\varphi'|^2)\partial_{\nu_g}\varphi\right] |z|^2
\\
&\hskip 1cm+\int_{\Sigma_0} \left[-2s\varphi'z'+2s\langle\nabla_g\varphi,\nabla_g z\rangle-s\square\varphi z\right]\partial_{\nu_g}z
\\
&\hskip 4cm +\int_{\Sigma_0} \left[s\partial_{\nu_g}\varphi |z'|^2-s\partial_{\nu_g}\varphi|\nabla_g z|_g^2\right],
\\
&\mathcal{B}_\Sigma:=\int_\Sigma \left[s^3(|\nabla_g\varphi|_g^2-|\varphi'|^2)\partial_{\nu_g}\varphi\right] |z|^2
\\
&\hskip 1cm+\int_\Sigma \left[-2s\varphi'z'+2s\langle\nabla_g\varphi,\nabla_g z\rangle-s\square\varphi z\right]\partial_{\nu_g}z
\\
&\hskip 4cm+\int_\Sigma \left[s\partial_{\nu_g}\varphi |z'|^2-s\partial_{\nu_g}\varphi|\nabla_g z|_g^2\right],\end{align*}
and
\begin{align*}&\mathcal{B}_T:=\int_{D\times\{T\}}\left[(-s/2)\partial_t\square\varphi-s^3\varphi'(|\nabla_g\varphi|_g^2-|\varphi'|^2)\right]|z|^2
\\
&\hskip 1cm+\int_{D\times\{T\}}\left[-2s\langle\nabla_g\varphi,\nabla_g z\rangle+s\square\varphi z\right]z'
\\
&\hskip 1.5cm+\int_{D\times\{T\}}s\varphi'\left[|z'|^2+|\nabla_g z|_g^2\right]\end{align*}
so that
\[
\mathcal{B}=\mathcal{B}_{\Sigma_0}+\mathcal{B}_\Sigma+\mathcal{B}_T.
\]
We check that
\[
\mathcal{B}_{\Sigma_0}\le C\int_{\Sigma_0}\left[s|z'|^2+s|\nabla_g z|_g^2+s^3|z|^2\right],
\]
and
\[\mathcal{B}_T\le C\int_{D\times\{T\}} \left[s|z'|^2+s|\nabla_g z|_g^2+s^3|z|^2\right].\]

On the other hand,  we have
\[
|\nabla_g z|_g^2=|\nabla_{\tau_g}z|_g^2+|\partial_{\nu_g} z|^2,\quad \langle\nabla_g\varphi,\nabla_g z\rangle=-|\nabla_g\phi|_g\partial_{\nu_g} z\quad \mbox{on}\; \Sigma
\]
and $\nu_g=-\frac{\nabla_g\phi}{|\nabla_g\phi|_g}$ on $\Sigma$. Whence
\begin{align*}
\mathcal{B}_\Sigma&=\int_\Sigma \left[-s^3(|\nabla_g\phi|_g^2-\gamma^2)|\nabla_g\phi|_g\right] |z|^2
\\
&\hskip 1cm+\int_\Sigma \left[2\gamma sz'-2s|\nabla_g\phi|_g\partial_{\nu_g} z-s\square\varphi z\right]\partial_{\nu_g}z
\\
&\hskip 1.5cm+\int_\Sigma \left[-s|\nabla_g\phi|_g|z'|^2+s|\nabla_g\phi|_g(|\nabla_{\tau_g}z|_g^2+|\partial_{\nu_g} z|^2)\right]
\\
&=\int_\Sigma \left[-s|\nabla_g\phi|_g|z'|^2+2\gamma s\partial_{\nu_g} z z'-s|\nabla_g\phi|_g|\partial_{\nu_g} z|^2+s|\nabla_g\phi|_g|\nabla_{\tau_g} z|_g^2\right.
\\
&\hskip 2cm\left.-s\square\varphi \partial_{\nu_g} z z-s^3(|\nabla_g\phi|_g^2-\gamma^2)|\nabla_g\phi|_g|z|^2\right]
\\
&\le\int_\Sigma s\left[-\delta|z'|^2+(2\gamma z'-\square\varphi z)\partial_{\nu_g} z-\delta|\partial_{\nu_g} z|^2+|\nabla_g\phi|_g|\nabla_{\tau_g}z|_g^2\right.
\\
&\hskip 6cm \left.-s^2(\delta^2-\gamma^2)\delta|z|^2\right],
\end{align*}
which means
\begin{align*}
&\mathcal{B}_\Sigma+\frac{\delta-\gamma}{2}\int_\Sigma s\left[|z'|^2+|\partial_{\nu_g}z|^2\right]
\\
&\hskip 1cm \le\int_\Sigma s\left[-\frac{\delta+\gamma}{2}|z'|^2+(2\gamma z'-\square\varphi z)\partial_{\nu_g} z-\frac{\delta+\gamma}{2}|\partial_{\nu_g} z|^2\right.
\\
&\hskip 5cm \left.+|\nabla_g\phi|_g|\nabla_{\tau_g}z|_g^2-s^2(\delta^2-\gamma^2)\delta|z|^2\right].
\end{align*}
Set $\eta:=\frac{\delta+\gamma}{2}$. By $0<\gamma<\delta$, we note that $\gamma<\eta<\delta$ and
\begin{align*}
&-\eta|z'|^2+(2\gamma z'-\square\varphi z)\partial_{\nu_g}z-\eta|\partial_{\nu_g} z|^2
\\
&\hskip 3cm =-\eta\left|z'-\frac{\gamma}{\eta}\partial_{\nu_g} z\right|^2-\frac{\eta^2-\gamma^2}{\eta}|\partial_{\nu_g} z|^2-\square\varphi z\partial_{\nu_g} z\\
&\hskip 3cm \le-\frac{\eta^2-\gamma^2}{\eta}\left|\partial_{\nu_g} z+\frac{\eta\square\varphi z}{2(\eta^2-\gamma^2)}\right|^2+C|z|^2
\\
&\hskip 3cm \le C|z|^2.\end{align*}
Therefore, we obtain
\[
\mathcal{B}_\Sigma+\frac{\delta-\gamma}{2}\int_\Sigma s\left[|z'|^2+|\partial_{\nu_g}z|^2\right]\le\int_\Sigma s\left[C|\nabla_{\tau_g}z|_g^2+C|z|^2-s^2(\delta^2-\gamma^2)\delta|z|^2\right]
\]
and then
\begin{align*}
&\mathcal{B}_\Sigma+\frac{\delta-\gamma}{2}\int_\Sigma s\left[|z'|^2+|\partial_{\nu_g}z|^2\right]+\frac{(\delta^2-\gamma^2)\delta}{2}\int_\Sigma s^3|z|^2
\\
&\hskip 7cm \le C\int_\Sigma s|\nabla_{\tau_g}z|_g^2,\quad s\ge s_\ast.
\end{align*}

Combining the estimates above, we get 
\begin{align*}
&C\left(\int_{Q}(s|\nabla_g z|_g^2+s^3|z|^2)+\int_{\Sigma}(s|z'|^2+s|\partial_{\nu_g}z|^2+s^3|z|^2)\right)
\\
&\hskip 1cm \le  (P_s^+z,P_s^-z)_g+\int_{Q}s\left[\langle\nabla_g\square\varphi,\nabla_g z\rangle+(1/2)\partial_t^2\square\varphi z\right]z
\\
&\hskip 6cm+\mathcal{B}_{\Sigma_0}+\mathcal{B}_T+\int_\Sigma s|\nabla_{\tau_g}z|_g^2
\\
&\hskip 1cm \le \|P_s z\|_g^2+\int_{Q} \left(|\nabla_g z|_g^2+s^2|z|^2\right)+\int_{\Sigma_0}\left(s|z'|^2+s|\nabla_g z|_g^2+s^3|z|^2\right)
\\
&\hskip 2cm+\int_\Sigma s|\nabla_{\tau_g}z|_g^2+\int_{D\times\{T\}}\left(s|z'|^2+s|\nabla_g z|_g^2+s^3|z|^2\right),\quad s\ge s_\ast.
\end{align*}
As the second term in the right-hand side  can absorbed by the left-hand side, we find
\begin{align*}
&C\left(\int_{Q}(s|\nabla_g z|_g^2+s^3|z|^2)+\int_\Sigma (s|z'|^2+s|\partial_{\nu_g}z|^2+s^3|z|^2)\right)
\\
&\hskip .5cm \le \|P_s z\|_g^2+\int_{\Sigma_0}(s|z'|^2+s|\nabla_g z|_g^2+s^3|z|^2)
\\
&\hskip 1cm+\int_\Sigma s|\nabla_{\tau_g}z|_g^2+\int_{D\times\{T\}}(s|z'|^2+s|\nabla_g z|_g^2+s^3|z|^2),\quad s\ge s_\ast.
\end{align*}
Since $u=e^{-s\varphi}z$ and $\nabla_{\tau_g}u=\nabla_{\tau_g}z$ holds by $\nabla_{\tau_g}\phi=0$, we end up getting
\begin{align*}
&C\left(\int_{Q} e^{2s\varphi}(s|\nabla_g u|_g^2+s^3|u|^2)+\int_\Sigma  e^{2s\varphi}(s|u'|^2+s|\partial_{\nu_g}u|^2+s^3|u|^2)\right)
\\
&\hskip .5cm\le \int_{Q} e^{2s\varphi}|\square u|^2+\int_{\Sigma_0}e^{2s\varphi}(s|u'|^2+s|\nabla_g u|_g^2+s^3|u|^2)
\\
&\hskip 1cm+\int_\Sigma  e^{2s\varphi}s|\nabla_{\tau_g}u|_g^2+\int_{D\times\{T\}} e^{2s\varphi}(s|u'|^2+s|\nabla_g u|_g^2+s^3|u|^2), \quad s\ge s_\ast.
\end{align*}
The proof is then complete.
\end{proof}

\appendix

\section{Analysis of the IBVP}\label{appA}

Let $A\colon D(A)\subset L^2(U)\rightarrow L^2(U)$ be the unbounded operator given by
\[
Au=\Delta_g u\quad \mbox{and}\quad D(A):=H_0^1(U)\cap H^2(U).
\]
Let $u\in D(A)$. By Green's formula, we have
\[
\int_U\Delta_gu\overline{u}dV_g=-\int_U|\nabla_g u|_g^2dV_g\le 0,
\]
that is, $A$ is dissipative. Let $\lambda >0$ and $u_0\in L^2(U)$. Since 
\[
B(u,v)=\lambda  \int_Uu\overline{v}dV_g+\int_U\langle \nabla_gu,\nabla_g\overline{v}\rangle dV_g,\quad u,v\in H_0^1(U)
\]
is an equivalent scalar product on $H_0^1(U)$, Riesz representation theorem shows that there exists a unique $u\in H_0^1(U)$ such that
\[
\lambda  \int_Uu\overline{v}dV_g+\int_U\langle \nabla_gu,\nabla_g\overline{v}\rangle dV_g=\int_Uu_0\overline{v}dV_g,\quad v\in H_0^1(U).
\]
Hence, $(\lambda -\Delta_g)u=u_0$ in the distributional sense and then $u\in D(A)$ according to the elliptic regularity (e.g., \cite[Theorem 9.25]{Brezis2011} and Remark 24 just after \cite[Theorem 9.26]{Brezis2011}).

We endow the Hilbert space $\mathcal{H}:=H_0^1(U)\times L^2(U)$ with the scalar product
\[
\langle (v_1,w_1)|(v_2,w_2)\rangle =\int_U [\langle \nabla_g v_1,\nabla_g \overline{v_2}\rangle+w_1\overline{w_2}]dV_g,\quad (v_j,w_j)\in \mathcal{H},\; j=1,2.
\] 
Define the unbounded operator $\mathcal{A}\colon D(\mathcal{A})\subset\mathcal{H}\rightarrow \mathcal{H}$ as follows
\[
\mathcal{A}(u,v)=(v,Au)\quad \mbox{and}\; D(\mathcal{A}):=\mathcal{H}_0:=(H_0^1(U)\cap H^2(U))\times H_0^1(U).
\]
Let $(u,v)\in D(\mathcal{A})$. As
\begin{align*}
\langle \mathcal{A}(u,v)|(u,v)\rangle &=\int_U \langle \nabla_g\overline{u},\nabla_gv\rangle dV_g+\int_U \Delta_gu\overline{v}dV_g
\\
&=\int_U \langle \nabla_g\overline{u},\nabla_gv\rangle dV_g-\int_U \langle \nabla_gu,\nabla_g\overline{v}\rangle dV_g,
\end{align*}
we obtain
\[
\Re \langle \mathcal{A}(u,v)|(u,v)\rangle =0,
\]
that is, $\mathcal{A}$ is dissipative.

Let $\lambda>0$ and $(u_0,u_1)\in \mathcal{H}$ and consider the equation
\[
\lambda (u,v)- \mathcal{A}(u,v)=(u_0,u_1).
\]
This equation is equivalent to the following system
\begin{equation}\label{we3}
(\lambda^2 -A)u=u_1+\lambda u_0,\quad v=\lambda u-u_0.
\end{equation}
We have seen above that $A$ is $m$-dissipative. Hence, \eqref{we3} admits unique solution $(u,v)\in D(\mathcal{A})$. In other words, we proved that $\mathcal{A}$ is $m$-dissipative and therefore, according to Lumer-Phillips theorem, $\mathcal{A}$ generates a contraction semigroup (e.g., \cite[Theorem 3.8.4]{Tucsnak2009}). 

Let $\mathcal{X}_0$ be the subspace of $\mathcal{X}$ given by
\[
\mathcal{X}_0=C^2([0,T];L^2(U))\cap C^1([0,T];H_0^1(U))\cap C([0,T];H_0^1(U)\cap H^2(U)).
\]

As $\mathcal{A}$ is the generator of a contraction semigroup, for each $(u_0,u_1)\in \mathcal{H}_0$ and $F\in W^{1,1}((0,T);L^2(U))$ the IBVP 
\[
\begin{cases}
(\partial_t^2-\Delta_g )v=F\quad \mbox{in}\; U\times (0,T),
\\
v_{|\Sigma}=0,
 \\
 (v(\cdot,0),\partial_tv(\cdot,0))=(u_0,u_1)
\end{cases}
\]
has a unique solution $v=v(u_0,u_1,F)\in \mathcal{X}_0$. Furthermore, we have
\begin{equation}\label{we2}
\|v\|_{\mathcal{X}}\le \mathbf{c}_0 \left(\|(u_0,u_1)\|_{\mathcal{H}_0}+\|F\|_{W^{1,1}((0,T);L^2(U))}\right),
\end{equation}
where $\mathbf{c}_0=\mathbf{c}_0(g,\phi,T)>0$ is a constant.

Next, consider the IBVP
\begin{equation}\label{exIBVP_p=0}
\begin{cases}
(\partial_t^2 -\Delta_g)u=0\quad \text{in}\; U\times(0,T),
\\
u_{|\Sigma}=f,
\\
(u(\cdot,0),\partial_tu(\cdot,0))=(u_0,u_1).
\end{cases}
\end{equation}
Let $(f,u_0,u_1)\in \mathcal{Z}$ and $H\in \dot{\overbrace{(f,u_0,u_1)}}$. If $F:=-(\partial_t^2 -\Delta_g)H$ ($\in W^{1,1}((0,T);L^2(U))$)  then we check that $u=u(f,u_0,u_1):=H+v(0,0,F)\in \mathcal{X}$ is the unique solution of 
\eqref{exIBVP_p=0}. Furthermore, as  $H\in \dot{\overbrace{(f,u_0,u_1)}}$ was taken arbitrary, \eqref{we2} implies
\[
\|u\|_{\mathcal{X}}\le \mathbf{c} \|(f,u_0,u_1)\|_{\mathcal{Y}},
\]
where $\mathbf{c}=\mathbf{c}(g,\phi,T)>0$ is a constant.

Let us remark that, according to perturbation theory of strongly continuous semigroups (e.g., \cite[Chapter 3.1]{Pazy1983}), the result for the IBVP \eqref{exIBVP_p=0} still holds when $\Delta_g u$ is replaced by $\Delta_g u+pu+q\partial_t u$, where $p,q\in L^\infty (U)$.

\section{Proof of Proposition \ref{energy_estimate}}\label{appB}

This section is devoted to the proof of Proposition \ref{energy_estimate}.

\begin{proof}
Throughout the proof, we fix $t\in(0,T)$ almost everywhere.

Multiplying $2\partial_t \overline{u}$ to the equation $(\partial_t^2-\Delta_g-p)u=0$, integrating it over $D$ and taking the real part of the equality yield
\begin{align*}
0&=\frac{d}{dt}\int_D|\partial_t u|^2dV_g-\Re\int_D2\Delta_g u\partial_t \overline{u}dV_g-\Re\int_D2pu\partial_t \overline{u}dV_g
\\
&=\frac{d}{dt}\int_D(|\partial_t u|^2+|\nabla_g u|_g^2)dV_g-\Re\int_{\partial D} 2\partial_{\nu_g} u\partial_t \overline{u}dS_g-\Re\int_D2pu\partial_t \overline{u}dV_g.
\end{align*}
Recall that we set
\[
\mathcal{E}(u)(t)=\int_{D\times\{t\}}\left(|\partial_t u|^2+|\nabla_g u|_g^2+|u|^2\right)dV_g.
\]
Then, we have trivially
\[
\left|\Re\int_D2pu\partial_t \overline{u}dV_g\right|\le \|p\|_{L^\infty(D)}\mathcal{E}(u)(t)
\]
and
\[
\frac{d}{dt}\int_D|u|^2dV_g\le \int_D (|\partial_t u|^2+|u|^2)dV_g\le \mathcal{E}(u)(t).
\]
In consequence, we have
\[
\mathcal{E}(u)'(t)\le \left(1+\|p\|_{L^\infty(D)}\right)\mathcal{E}(u)(t)+2\int_{\partial D} |\partial_{\nu_g}u||\partial_t u|dV_g.
\]
Gr\"{o}nwall's inequality gives
\begin{align*}
&\mathcal{E}(u)(t)\le 2e^{(1+\|p\|_{L^\infty(D)})t}\int_{\partial D\times (0,t)}|\partial_{\nu_g} u||\partial_t u|dS_gdt
\\
&\hskip 4cm\le e^{(1+\|p\|_{L^\infty(D)})t}\left(\|u\|_{H^1(\Upsilon)}^2+\|\partial_{\nu_g} u\|_{L^2(\Upsilon)}^2\right).
\end{align*}
By the density argument, this inequality holds for all $t\in[0,T]$.
\end{proof}

\section*{Acknowledgement}

This work was supported by JSPS KAKENHI Grant Number JP23KK0049.

\section*{Declarations}
\subsection*{Conflict of interest}
The authors declare that they have no conflict of interest.

\subsection*{Data availability}
Data sharing not applicable to this article as no datasets were generated or analyzed during the current study.

\bibliographystyle{plain}
\bibliography{hyperbolic_continuation_8}

\end{document}